\documentclass[11pt]{amsart}
\usepackage{amsmath, amssymb}
\usepackage{amscd}
\usepackage{verbatim}
\usepackage{amssymb,latexsym,amsmath,amsfonts}
\usepackage{latexsym}
\usepackage[mathscr]{eucal}
\usepackage{bm}

\title[From Stinespring dilation to Sz.-Nagy dilation on $\Gamma$]
{From Stinespring dilation to Sz.-Nagy dilation on the symmetrized
bidisc and operator models}

\author{Sourav Pal}
\address{Department of Mathematics, Ben-Gurion University of the Negev, Be'er Sheva-84105, Israel.}
\email{sourav@math.bgu.ac.il}
\thanks{The author was supported in part by a postdoctoral fellowship of the Skirball
Foundation via the Center for Advanced Studies in Mathematics at
Ben-Gurion University of the Negev.}



\def ~{\hspace{1mm}}

\def\textmatrix#1&#2\\#3&#4\\{\bigl({#1 \atop #3}\ {#2 \atop #4}\bigr)}
\def\dispmatrix#1&#2\\#3&#4\\{\left({#1 \atop #3}\ {#2 \atop #4}\right)}
\newcommand{\beg}{\begin{equation}}
\newcommand{\eeg}{\end{equation}}
\newcommand{\ben}{\begin{eqnarray*}}
\newcommand{\een}{\end{eqnarray*}}

\newtheorem{thm}{Theorem}[section]
\newtheorem{cor}[thm]{Corollary}
\newtheorem{lem}[thm]{Lemma}

\newtheorem{prop}[thm]{Proposition}
\numberwithin{equation}{section}
\theoremstyle{definition}
\newtheorem{defn}[thm]{Definition}
\newtheorem{rem}[thm]{Remark}

\def\textmatrix#1&#2\\#3&#4\\{\bigl({#1 \atop #3}\ {#2 \atop #4}\bigr)}
\def\dispmatrix#1&#2\\#3&#4\\{\left({#1 \atop #3}\ {#2 \atop #4}\right)}

\begin{document}

\begin{abstract}

We provide an explicit normal distinguished boundary dilation to a
pair of commuting operators $(S,P)$ having the closed symmetrized
bidisc $\Gamma$ as a spectral set. This is called Sz.-Nagy
dilation of $(S,P)$. The operator pair that dilates $(S,P)$ is
obtained by an application of Stinespring dilation of $(S,P)$
given by Agler and Young. We further prove that the dilation is
minimal and the dilation space is no bigger than the dilation
space of the minimal unitary dilation of the contraction $P$. We
also describe model space and model operators for such a pair
$(S,P)$.

\end{abstract}

\maketitle

\tableofcontents

\section{Introduction}

The {\em closed symmetrized bidisc} and its {\em distinguished
boundary}, denoted by $\Gamma$ and $b\Gamma$ respectively, are
defined by
\begin{align*}
\Gamma &= \{(z_1+z_2,z_1z_2):\; |z_1|\leq 1, |z_2|\leq 1
\}\subseteq \mathbb C^2 \qquad \text{and}\\ b\Gamma &=\{
(z_1+z_2,z_1z_2):\; |z_1|=|z_2|=1 \}\subseteq \Gamma.
\end{align*}
Clearly, the points of $\Gamma$ and $b\Gamma$ are the
symmetrization of the points of the closed bidisc
$\overline{\mathbb D}^2$ and the torus $\mathbb T^2$ respectively,
where the symmetrization map is the following:
\[
\pi: \mathbb
C^2 \rightarrow \mathbb C^2, \quad (z_1,z_2)\mapsto
(z_1+z_2,z_1z_2).
\]

Function theory, hyperbolic geometry and operator theory related
to the set $\Gamma$ have been well studied over past three decades
(e.g. \cite{ay-jfa, ay-ems, ay-blm, ay-jot, ay-jga, ay-caot,
tirtha-sourav, tirtha-sourav1, pal-shalit}).

\begin{defn}
A pair of commuting operators $(S,P)$, defined on a Hilbert space
$\mathcal H$, that has $\Gamma$ as a spectral set is called a
$\Gamma$-\textit{contraction}, i.e. the joint spectrum $\sigma
(S,P)\subseteq \Gamma$ and
\[
 \|f(S,P)\|\leq
\sup_{(z_1,z_2)\in\Gamma}|f(z_1,z_2)|,
\]
for all rational
functions $f$ with poles off $\Gamma$.
\end{defn}
By virtue of polynomial convexity of $\Gamma$, the definition can
be made more precise by omitting the condition on joint spectrum
and by replacing rational functions by polynomials. It is clear
from the definition that if $(S,P)$ is a
$\Gamma$-contraction then so is $(S^*,P^*)$ and $\|S\|\leq 2$,\, $\|P\|\leq 1$. \\

A commuting $d$-tuple of operators
$\underline{T}=(T_1,T_2,\cdots,T_d)$ for which a particular subset
of $\mathbb C^d$ is a spectral set, has been studied for a long
time and many important results have been obtained (see
\cite{paulsen}). Let $W\subseteq \mathbb C^d$ be a spectral set
for $\underline{T}=(T_1,T_2,\cdots,T_d)$. A {\em normal
$bW$-dilation} of $\underline{T}$ is a commuting $d$-tuple of
normal operators $\underline{N}=(N_1,\cdots,N_d)$ defined on a
larger Hilbert space $\mathcal K \supseteq \mathcal H$ such that
the joint spectrum $\sigma(\underline{N})\subseteq bW$ and
$q(\underline{T})=P_{\mathcal H}q(\underline{N})\big|_{\mathcal
H}$, for any polynomial $q$ in $d$-variables $z_1,\dots,z_d$. A
celebrated theorem of Arveson states that $W$ is a complete
spectral set for $\underline T$ if and only if $\underline T$ has
a normal $bW$-dilation, (Theorem 1.2.2 and its corollary,
\cite{arveson2}). Therefore, a necessary condition for $\underline
T$ to have a normal $bW$-dilation is that $W$ be a spectral set
for $\underline T$. Sufficiency has been investigated for several
domains in several contexts, and it has been shown to have a
positive answer when $W=\overline{\mathbb D}$ \cite{nazy}, when
$W$ is an annulus \cite{agler-ann}, when $W=\overline{\mathbb
D^2}$ \cite{ando} and when $W=\Gamma$ \cite{ay-jfa}. Also we have
failure of rational dilation on a triply connected domain in
$\mathbb C$
\cite{ahr, DM}.\\

The main aim of this paper is to construct an explicit normal
$b\Gamma$-dilation to a $\Gamma$-contraction $(S,P)$, which we
call Sz.-Nagy dilation of $(S,P)$. As a consequence we obtain a
concrete functional model for $(S,P)$. The principal source of
inspiration is the following dilation theorem which will be called
Stinespring dilation of $(S,P)$.
\begin{thm}[Agler and Young, \cite{ay-jfa}] \label{stinespring}
Let $(S,P)$ be a pair of commuting operators on a Hilbert space
$\mathcal H$ such that the joint spectrum $\sigma(S,P)\subseteq
\Gamma$. The following are equivalent.
\begin{enumerate}
\item $(S,P)$ is a $\Gamma$-contraction; \item $\rho(\alpha
S,\,{\alpha}^2P)\geq 0,$ for all $\alpha\in\mathbb D$, where
\[
\rho(S,P)=2(I-P^*P)-(S-S^*P)-(S^*-P^*S);
\]
\item for every matrix
polynomial $f$ in two variables
\[
\|f(S,P)\|\leq \sup_{z\in\Gamma}\|f(z)\|;
\]
\item there exist
Hilbert spaces $\mathcal H_-,\,\mathcal H_+$ and commuting normal
operators $\tilde S,\tilde P$ on $\mathcal K=\mathcal H_-\oplus
\mathcal H \oplus \mathcal H_+$ such that the joint spectrum
$\sigma(\tilde S, \tilde P)$ is contained in the distinguished
boundary of $\Gamma$ and $\tilde S, \tilde P$ are expressible by
operators matrices of the form
\[
\tilde S=\begin{bmatrix} *&*&*\\0&S&*\\0&0&* \end{bmatrix},
\quad \tilde P=\begin{bmatrix}*&*&*\\0&P&*\\0&0&* \end{bmatrix}
\]
with respect to the orthogonal decomposition $\mathcal K=\mathcal
H_-\oplus \mathcal H\oplus \mathcal H_+$.
\end{enumerate}

\end{thm}
The reason of calling it Stinespring dilation is that part-(4) of
the above theorem was obtained by an application of Stinespring's
theorem (\cite{paulsen}). For a
proof of the above theorem, see Theorem 1.2 in \cite{ay-jfa}.\\

In Theorem \ref{main-dilation-theorem}, which is the main result
of this paper, we provide such Hilbert spaces $\mathcal H_-,\,
\mathcal H_+$ and such operators $\tilde S, \tilde P$ explicitly.
Indeed, the dilation space $\mathcal K \,(=\mathcal H_-\oplus
\mathcal H \oplus \mathcal H_+)$ can be chosen to be $l^2(\mathcal
D_P)\oplus \mathcal H \oplus l^2(\mathcal D_{P^*})$ which is same
as the dilation space of the minimal unitary dilation of $P$ and
the operator $\tilde P$ can precisely be the minimal unitary
dilation of $P$. Here $\mathcal D_P=\overline{\textup{Ran}}\,D_P
,\textup{ where } D_P=(I-P^*P)^{\frac{1}{2}}$. In order to
construct an operator for $\tilde S$, i.e. to remove the stars
from the matrix of $\tilde S$, we need a couple of operators
$F,\,F_*$ which turned out to be the unique solutions to the
operator equations
\begin{align}\label{fundamental:eqn}
S-S^*P &=D_PXD_P, \quad X\in\mathcal L(\mathcal D_P) \\ \nonumber
S^*-SP^* &=D_{P^*}X_*D_{P^*}, \quad X_*\in \mathcal L(\mathcal
D_{P^*})
\end{align}
respectively (Theorem \ref{fundamentalop}). Such an operator
equation (\ref{fundamental:eqn}) was solved in
\cite{tirtha-sourav} (Theorem 4.1 in \cite{tirtha-sourav})
independently to get a $\Gamma$-isometric dilation of a
$\Gamma$-contraction (Theorem 4.3 in \cite{tirtha-sourav}) but it
was not a normal $b\Gamma$-dilation. The unique operators $F$ and
$F_*$ were called the \textit{fundamental operators} of the
$\Gamma$-contractions $(S,P)$ and $(S^*,P^*)$ respectively. The
fundamental operators of $(S,P)$ and $(S^*,P^*)$ play the key role
in the construction of the operator that works for $\tilde S$.
Since the dilation space is precisely the space of minimal unitary
dilation of $P$, the dilation naturally becomes minimal. This is
somewhat surprising because it is a dilation in several variables.\\

As the title of the paper indicates, we obtain Sz.-Nagy dilation
of a $\Gamma$-contraction $(S,P)$ from its Stinespring dilation in
the sense that we obtain the key ingredient in the dilation, the
fundamental operator, as a consequence of Stinespring dilation.
Indeed, Theorem \ref{stinespring} leads to the following model for
$\Gamma$-contractions (Theorem 3.2 in \cite{ay-jot}).
\begin{thm}[Agler and Young \cite{ay-jot}]\label{model}
Let $(S,P)$ be a $\Gamma$-contraction on a Hilbert space $\mathcal
H$. There exists a Hilbert space $\mathcal K$ containing $\mathcal
H$ and a $\Gamma$-co-isometry $(S^{\flat},P^{\flat})$ on $\mathcal
K$ and an orthogonal decomposition $\mathcal K_1\oplus \mathcal
K_2$ of $\mathcal K$ such that:
\begin{enumerate}
\item[(i)] $\mathcal H$ is a common invariant subspace of
$S^{\flat}$ and $P^{\flat}$ and $S=S^{\flat}|_{\mathcal
H},\;P=P^{\flat}|_{\mathcal H}$; \item[(ii)] $\mathcal K_1$ and
$\mathcal K_2$ reduce both $S^{\flat}$ and $P^{\flat}$;
\item[(iii)] $(S^{\flat}|_{\mathcal K_1},P^{\flat}|_{\mathcal
K_1})$ is a $\Gamma$-unitary; \item[(iv)] there exists a Hilbert
space $E$ and an operator $A$ on $E$ such that $\omega(A)\leq 1$
and $(S^{\flat}|_{\mathcal K_2},P^{\flat}|_{\mathcal K_2})$ is
unitarily equivalent to $(T_{\psi},T_z)$ acting on $H^2(E)$, where
$\psi \in\mathcal L(E)$ is given by $\psi(z)=A^*+A\bar z, \quad
z\in \bar{\mathbb D}$.
\end{enumerate}
\end{thm}
In section 3, we establish the existence and uniqueness of
fundamental operator $F$ of $(S,P)$ (Theorem \ref{fundamentalop})
by an application of Theorem \ref{model}. Moreover, we show that
the numerical radius of $F$ is not greater than 1.\\

In section 5, we describe a functional model for
$\Gamma$-contractions (Theorem \ref{model2}), which can be treated
as a concrete formulation of the model given as Theorem
\ref{model} above. We specify the model space and model operators.
Also a model is provided for a \textit{pure}
$\Gamma$-\textit{isometry} $(\hat S,\hat P)$ in terms of Toeplitz
operators $(T_{\varphi},T_z)$ defined on the vectorial Hardy space
$H^2(\mathcal D_{{\hat P}^*})$, where the multiplier function is
given as $\varphi(z)={\hat F}_*^*+{\hat F}_*z$, ${\hat F}_*$ being
the fundamental operator of $(\hat S,\hat P)$. This model is
obtained independently in a simpler way without an application of
the functional model for pure $\Gamma$-contractions (see Theorem
3.1 in \cite{tirtha-sourav1}). Let us mention that the class of
pure $\Gamma$-isometries parallels the class of pure isometries in
one variable operator
theory.\\

In section 2, we recall some preliminary results from the
literature of $\Gamma$-contraction and these results will be used
in sequel.

\section{Preliminary results on $\Gamma$-contractions}

In the literature of $\Gamma$-contraction, \cite{ay-jfa, ay-ems,
ay-blm, ay-jot}, there are special classes of
$\Gamma$-contractions like $\Gamma$-unitaries,
$\Gamma$-isometries, $\Gamma$-co-isometries which are analogous to
unitaries, isometries and co-isometries of single variable
operator theory.

\begin{defn}\label{dboundary} A commuting pair $(S,P)$, defined on a Hilbert space $\mathcal H$, is
called a $\Gamma$-{\em unitary} if $S$ and $P$ are normal
operators and $\sigma(S,P)$ is contained in the distinguished
boundary $b\Gamma$.
\end{defn}
\begin{defn}
A commuting pair $(S,P)$ is called a $\Gamma$-{\em isometry} if it
the restriction of $\Gamma$-unitary to a joint invariant subspace,
i.e. a $\Gamma$-isometry is a pair of commuting operators which
can be extended to a $\Gamma$-unitary.
\end{defn}
\begin{defn}
A $\Gamma$-{\em co-isometry} is the adjoint of a
$\Gamma$-isometry, i.e. $(S,P)$ is a $\Gamma$-co-isometry if
$(S^*,P^*)$ is a $\Gamma$-isometry.
\end{defn}
\begin{defn}
A $\Gamma$-isometry $(S,P)$ is said to be {\em pure} if $P$ is a
pure isometry. A {\em pure} $\Gamma$-{\em co-isometry} is the
adjoint of a pure $\Gamma$-isometry.
\end{defn}
\begin{defn}
Let $(S,P)$ be a $\Gamma$-contraction on a Hilbert space $\mathcal
H$. A commuting pair $(T,V)$ defined on $\mathcal K$ is said to be
a $\Gamma$-isometric (or $\Gamma$-unitary) extension if $\mathcal
H\subseteq \mathcal K$, $(T,V)$ is a $\Gamma$-isometry (or a
$\Gamma$-unitary) and $T|_{\mathcal H}=S,\, V|_{\mathcal H}=P$.
\end{defn}
\noindent We are now going to state some useful results on
$\Gamma$-contractions without proofs because the proofs are either
routine or could be found out in \cite{ay-jfa} and \cite{ay-jot}.
\begin{prop}
If $T_1,T_2$ are commuting contractions then their symmetrization
$(T_1+T_2,T_1T_2)$ is a $\Gamma$-contraction.
\end{prop}
Note that, all $\Gamma$-contractions do not arise as a
symmetrization of two contractions. The following result
characterizes the $\Gamma$-contractions which can be obtained as a
symmetrization of two commuting contractions.
\begin{lem}[\cite{ay-jot}]
Let $(S,P)$ be a $\Gamma$-contraction. Then
$(S,P)=(T_1+T_2,T_1T_2)$ for a pair of commuting operators
$T_1,T_2$ if and only if $S^2-4P$ has a square root that commutes
with both $S$ and $P$.
\end{lem}

Here is a set of characterizations for $\Gamma$-unitaries.
\begin{thm}[\cite{ay-jot}] \label{G-unitary}
Let $(S,P)$ be a pair of commuting operators defined on a Hilbert
space $\mathcal{H}.$ Then the following are equivalent:
\begin{enumerate}

\item $(S,P)$ is a $\Gamma$-unitary\; ; \item there exist
    commuting unitary operators $U_{1}$ and $
    U_{2}$ on $\mathcal{H}$ such that
\[
S= U_{1}+U_{2},\quad P= U_{1}U_{2}\;
\]
\item $P$ is
unitary,\;$S=S^*P,\;$\;and $r(S)\leq2,$ \; where
    $r(S)$ is the spectral radius of $S$.
\end{enumerate}
\end{thm}

\noindent We now present a structure theorem for the class of
$\Gamma$-isometries.
\begin{thm}[\cite{ay-jot}] \label{Gamma-isometry}
Let $S,P$ be commuting operators on a Hilbert space $\mathcal{H}.$
The following statements are all equivalent:\begin{enumerate}
\item $(S,P)$ is a $\Gamma$-isometry; \item if $P$ has
Wold-decomposition with respect to the orthogonal decomposition
$\mathcal H=\mathcal H_1\oplus \mathcal H_2$ such that
$P|_{\mathcal H_1}$ is unitary and $P|_{\mathcal H_2}$ is pure
isometry then $\mathcal H_1,\,\mathcal H_2$ reduce $S$ also and
$(S|_{\mathcal H_1},P|_{\mathcal H_1})$ is a $\Gamma$-unitary and
$(S|_{\mathcal H_2},P|_{\mathcal H_2})$ is a pure
$\Gamma$-isometry; \item $P$ is an isometry\;,\;$S=S^*P$ and
$r(S)\leq2$.
\end{enumerate}
\end{thm}

\section{The fundamental operator of a $\Gamma$-contraction}

Let us recall that the {\em numerical radius} of an operator $T$
on a Hilbert space $\mathcal{H}$ is defined by
\[
\omega(T) = \sup \{|\langle Tx,x \rangle|\; : \;
\|x\|_{\mathcal{H}}\leq 1\}.
\]
It is well known that
\begin{eqnarray}\label{nradius}
r(T)\leq \omega(T)\leq \|T\| \quad \text{ and } \quad
\frac{1}{2}\|T\|\leq \omega(T)\leq \|T\|,
\end{eqnarray}
where $r(T)$ is the spectral radius of $T$. The following is an
interesting result about the numerical radius of an operator and
this will be used in this section.
\begin{lem} \label{basicnrlemma} The numerical radius of an operator $X$ is not greater than
1 if and only if  Re $\beta X \le I$ for all complex numbers
$\beta$ of modulus $1$.
\end{lem}

\begin{proof}
It is obvious that $\omega(X)\leq1$ implies that $\mbox{Re }\beta
X\leq I$ for all $\beta\in\mathbb T$. We prove the other way. By
hypothesis, $\langle$Re $ \beta X h,h\rangle \leq1$ for all
$h\in\mathcal{H}$ with $\|h\|\leq1$ and for all $\beta \in
\mathbb{T}$. Note that $\langle$Re $ \beta X h,h\rangle =$ Re $
\beta \langle  X h,h\rangle$. Write $\langle Xh,h \rangle = e^{i
\varphi_h}|\langle Xh,h\rangle|$ for some $\varphi_h
\in\mathbb{R},$ and then choose $\beta = e^{- i \varphi_h}$. Then
we get  $|\langle Xh,h \rangle|\leq1$ and this holds for each
$h\in\mathcal{H}$ with $\|h\|\leq1.$ Hence done. \end{proof}

We are going to prove the existence and uniqueness of solution to
the operator equation
\[
S-S^*P=D_PXD_P, \quad X\in \mathcal L(\mathcal D_P)
\]
by an application of a famous result due to Douglas, Muhly and
Pearcy. Let us again mention here that the same operator equation
has been solved in \cite{tirtha-sourav} (Theorem 4.2)
independently by using operator Fejer-Riesz Theorem. Here is the
famous result of Douglas, Muhly and Pearcy.

\begin{prop}[Douglas, Muhly and Pearcy, \cite{DMP}]\label{Douglas}
For $i=1,2$, let $T_i$ be a contraction on a Hilbert space
$\mathcal H_i$, and let $X$ be an operator mapping $\mathcal H_2$
into $\mathcal H_1$. A necessary and sufficient condition that the
operator on $\mathcal H_1\oplus \mathcal H_2$ defined by the
matrix
\[
\begin{pmatrix} T_1&X\\0&T_2 \end{pmatrix}
\]
be a contraction is that there exist a contraction $C$ mapping
$\mathcal H_2$ into $\mathcal H_1$ such that
\[
X=\sqrt{I_{\mathcal
H_1}-T_1T_1^*}C\sqrt{I_{\mathcal H_2}-T_2^*T_2}.
\]
\end{prop}

\begin{thm}[Existence and Uniqueness]\label{fundamentalop}
For a $\Gamma$-contraction $(S,P)$ defined on $\mathcal H$, the
operator equation
\[
S-S^*P=D_PXD_P
\]
has a unique solution $F$ in $\mathcal L(\mathcal
D_P)$ and $\omega(F)\leq 1$.
\end{thm}
\begin{proof}
By Theorem \ref{model}, there is a $\Gamma$-co-isometry $(T,V)$ on
a larger Hilbert space $\mathcal K\supseteq \mathcal H$ such that
$\mathcal H$ is a joint invariant subspace of $T$ and $V$ and
\[
S=T|_{\mathcal H},\; P=V|_{\mathcal H}.
\]
Also $\mathcal K$ has orthogonal decomposition $\mathcal K
=\mathcal K_1 \oplus \mathcal K_2$ and
\[
T=\begin{pmatrix}
 T_1&0\\0&T_2 \end{pmatrix}, \; V=\begin{pmatrix} V_1&0\\0&V_2
 \end{pmatrix} \textup{ on } \mathcal K=\mathcal K_1\oplus \mathcal
 K_2
\]
such that $(T_1,V_1)$ is a $\Gamma$-unitary and there is a Hilbert
space $E$ and a unitary $U_1:\mathcal K_2 \rightarrow H^2(E)$ such
that
\[
T_2^*=U_1^*T_{\varphi}U_1,\; V_2^*=U_1^*T_zU_1,
\]
where $\varphi(z)=A+A^*z$ for some $A\in \mathcal B(E)$ with
numerical radius of $A$ being not greater than $1$. Clearly
$T_2=U_1^*T_{\varphi}^*U_1$ and $V_2=U_1^*T_z^*U_1$. Again
$H^2(E)$ can be identified with $l^2(E)$ and consequently the
operator pair $(T_{\varphi},T_z)$ can be identified with
$(M_{\varphi},M_z)$, where $M_{\varphi}$ and $M_z$ are defined on
$l^2(E)$ in the following way:
\[
M_{\varphi}=\begin{bmatrix} A&0&0&\cdots \\
A^* &A&0&\cdots\\ 0&A^*&A&\cdots \\ \vdots& \vdots&\vdots&\ddots
\end{bmatrix}, \; M_z=\begin{bmatrix} 0&0&0&\cdots\\ I&0&0&\cdots\\ 0&I&0&\cdots\\
\vdots&\vdots&\vdots&\ddots \end{bmatrix}.
\]
Therefore we can say that there is a unitary $U:\mathcal
K_2\rightarrow l^2(E)$ such that $T_2=U^*M_{\varphi}^*U, \textup{
and } V_2=U^*M_z^*U$. Now

\begin{align*}
T_2-T_2^*V_2 &= U^*\begin{bmatrix} A&A^*&0&\cdots\\ 0&A&A^*&\cdots\\
0&0&A&\cdots\\ \vdots&\vdots&\vdots&\ddots
\end{bmatrix}U \\& \quad - U^*\begin{bmatrix} A^*&0&0&\dots\\ A&A^*&0&\cdots\\ 0&A&A^*&\cdots\\ \vdots&\vdots&\vdots&\ddots \end{bmatrix}
\begin{bmatrix} 0&I&0&\cdots \\ 0&0&I&\cdots \\ 0&0&0&\cdots \\ \vdots&\vdots&\vdots&\ddots
\end{bmatrix}U \\ &= U^*\begin{bmatrix} A&A^*&0&\cdots \\ 0&A&A^*&\cdots \\ 0&0&A&\cdots\\ \vdots&\vdots&\vdots&\ddots
\end{bmatrix}U
- U^*\begin{bmatrix} 0&A^*&0&\cdots \\ 0&A&A^*&\cdots \\
0&0&A&\cdots\\ \vdots&\vdots&\vdots&\ddots
\end{bmatrix}U \\ &= U^*\begin{bmatrix} A&0&0&\cdots \\ 0&0&0&\cdots \\ 0&0&0&\cdots \\
\vdots&\vdots&\vdots&\ddots \end{bmatrix}U.
\end{align*}

Also
\[
D_{V_2}^2=I-V_2^*V_2=U^*(I-M_zM_z^*)U=U^*
\begin{bmatrix} I&0&0&\cdots \\ 0&0&0&\cdots\\
0&0&0&\cdots \\ \vdots&\vdots&\vdots&\ddots
\end{bmatrix} U.
\]
It is merely said that $D_{V_2}^2=D_{V_2}$ and therefore if we set
\[
X=U^*\begin{bmatrix} A^*&0&0&\hdots \\ 0&0&0&\hdots \\ 0&0&0&\hdots \\
\vdots&\vdots&\vdots&\ddots
\end{bmatrix}U
\]
then $X\in\mathcal L(\mathcal D_{V_2})$ and $
T_2-T_2^*V_2=D_{V_2}XD_{V_2} $. Since $(T_1,V_1)$ is a
$\Gamma$-unitary, Theorem \ref{G-unitary} guarantees that
$T_1=T_1^*V_1$. Therefore,
\[
T-T^*V=\begin{bmatrix} T_1-T_1^*V_1&0\\ 0&T_2-T_2^*V_2
\end{bmatrix}=\begin{bmatrix} 0&0\\0&T_2-T_2^*V_2 \end{bmatrix}.
\]
Also
\[
D_V^2=\begin{bmatrix} I_{\mathcal
K_1}-V_1^*V_1&0\\0&I_{\mathcal K_2}-V_2^*V_2
\end{bmatrix}=\begin{bmatrix} 0&0\\0&I_{\mathcal K_2}-V_2^*V_2 \end{bmatrix}.
\]
Therefore, $\mathcal D_V=\mathcal D_{V_2}$ and $X$ satisfies the
relation $T-T^*V=D_VXD_V$. Also
\[
\|T-T^*V\|=\|D_VXD_V\|\leq \|X\|\leq 2, \quad \textup{ by relation
}(\ref{nradius})\textup{ as } \omega(X)\leq 1.
\]
Now consider the matrix
\[
J=\begin{bmatrix} V^*&\dfrac{T-T^*V}{2} \\
0&V \end{bmatrix}
\]
defined on $\mathcal K\oplus \mathcal K$. Since
\[
\frac{T-T^*V}{2}=D_V\frac{X}{2}D_V=(I-V^*V)^{\frac{1}{2}}\frac{X}{2}(I-V^*V)^{\frac{1}{2}},
\]
where $\dfrac{T-T^*V}{2}$ and $\dfrac{X}{2}$ are contractions, by
Proposition \ref{Douglas}, the matrix $J$ is a contraction. Again
let us consider another matrix $J_H$ defined on $\mathcal
H\oplus\mathcal H$ by
\[
J_H=\begin{bmatrix} P_{\mathcal H}V^*|_{\mathcal H}&P_{\mathcal
H}(\dfrac{T-T^*V}{2})|_{\mathcal H}
\\ 0& P_{\mathcal H}V|_{\mathcal H}
\end{bmatrix}.
\]
Since $(T,V)$ is $\Gamma$-co-isometric extension of $(S,P)$, we
have that
\[
J_H=\begin{bmatrix} P^*&\dfrac{S-S^*P}{2} \\ 0&P
\end{bmatrix}.
\]
For
\[
\begin{bmatrix} h_1\\h_2 \end{bmatrix}\in\mathcal H\oplus
\mathcal H,
\]
we have that
\begin{align*}
\|J_H\begin{bmatrix} h_1\\ h_2 \end{bmatrix}\|^2
&=\|\begin{bmatrix} P_{\mathcal H}V^*h_1+P_{\mathcal H} (
\dfrac{T-T^*V}{2} )h_2
\\ P_{\mathcal H}Vh_2\end{bmatrix}\|^2 \\& = \| P_{\mathcal H}(V^*h_1+ \frac{T-T^*V}{2}h_2)
\|^2+\|P_{\mathcal H}Vh_2 \|^2 \\& \leq \|
V^*h_1+\frac{T-T^*V}{2}h_2 \|^2+\| Vh_2 \|^2, \text{ since }
P_{\mathcal H} \text{ is a projection } \\& \leq \|
\begin{bmatrix}h_1 \\ h_2 \end{bmatrix} \|^2, \text{ since } J
\textup{ is a contraction}.
\end{align*}
Therefore, $J_H$ is a contraction. Applying Proposition
\ref{Douglas} again we get an operator $F\in\mathcal L(\mathcal
H)$ such that $\dfrac{F}{2}$ is a contraction and that
\[
\dfrac{S-S^*P}{2}=D_P\dfrac{F}{2}D_P.
\]
Obviously the domain of $F$
can be specified to be $\mathcal D_P\subseteq \mathcal H$. Hence
\[
S-S^*P=D_PFD_P
\]
where $F\in\mathcal L(\mathcal D_P)$ and the
existence of the fundamental operator of $(S,P)$ is guaranteed.\\

For uniqueness let there be two such solutions $F$ and $F_1$. Then
\[
{D}_P\tilde{F}{D}_P=0, \quad \text{where } \tilde{F}=F-F_1 \in
\mathcal L(\mathcal D_P).
\] Then
\[
\langle \tilde{F}{D}_Ph,{D}_Ph^{\prime} \rangle=\langle
{D}_P\tilde{F}{D}_Ph,h^{\prime} \rangle =0
\]
which shows that $\tilde{F}=0$ and hence $F=F_1$.\\

To show that the numerical radius of $F$ is not greater than 1,
note that $\rho(\alpha S,{\alpha}^2P)\geq 0$, for all $\alpha\in
\mathbb D$, by Theorem \ref{stinespring} and the inequality can be
extended by continuity to all points in $\overline{\mathbb D}$.
Therefore, in particular for $\beta \in \mathbb T$, we have
$$D_P^2\geq \textup{Re }\beta (S-S^*P) =\textup{Re }\beta
(D_PFD_P) $$ which implies that
$$ D_P(I_{\mathcal D_P} - \textup{Re }\beta F)D_P\geq 0.$$
Therefore,
\[
\langle ( I_{\mathcal D_P} - \text{Re }(\beta F )) D_P h , D_Ph
\rangle = \langle D_P (I_{\mathcal D_P} - \text{Re }( \beta F ))
D_P h , h \rangle \geq 0
\]
and consequently we obtain
\[
\text{Re } \beta F \leq I_{\mathcal D_P}, \textup{ for all
}\beta\in \mathbb T.
\]
Therefore by Lemma \ref{basicnrlemma}, the
numerical radius of $F$ is not greater than 1.

\end{proof}

\begin{rem}
\em{The fundamental operator of a $\Gamma$-isometry or a
$\Gamma$-unitary $(S,P)$ is the zero operator because $S=S^*P$ in
this case}.
\end{rem}
The following result is obvious and a proof to this can be found
in \cite{tirtha-sourav1}.
\begin{prop}\label{end-prop}
Let $(S,P)$ and $(S_1,P_1)$ be two $\Gamma$-contractions on a
Hilbert space $\mathcal H$ and let $F$ and $F_1$ be their
fundamental operators respectively. If $(S,P)$ and $(S_1,P_1)$ are
unitarily equivalent then so are $F$ and $F_1$.
\end{prop}

\section{Geometric construction of normal dilation}

 In this section, we present an explicit construction of a normal $b\Gamma$-dilation, i.e. a $\Gamma$-unitary
 dilation of a $\Gamma$-contraction. In the literature, the $\Gamma$-unitary and
$\Gamma$-isometric dilation of a $\Gamma$-contraction are defined
in the following way:
 \begin{defn}
 Let $(S,P)$ be a $\Gamma$-contraction on a Hilbert space $\mathcal H$. A pair of commuting
 operators $(T,U)$ defined on a Hilbert space $\mathcal K\supseteq \mathcal H$ is said to be
 a $\Gamma$-{\em unitary dilation} of $(S,P)$ if $(T,U)$ is a $\Gamma$-unitary and
 $P_{\mathcal H}(T^mU^n)|_{\mathcal H}=S^mP^n, \quad
 n=0,1,2,\dots$. Moreover, the dilation will be called {\em minimal} if
 \[
 \mathcal K=\overline{\text{span}}\{ T^mU^nh:\; h\in\mathcal H,\; m,n=0,\pm 1, \pm 2, \cdots\},
 \]
 where $T^{-m},U^{-n}$ for positive integers $m,n$ are defined
 as ${T^*}^{m}$ and ${U^*}^n$ respectively.
 A $\Gamma$-{\em isometric dilation} of a $\Gamma$-contraction
 is defined in a similar way where the word $\Gamma$-unitary is
 replaced by $\Gamma$-isometry. But when we talk about
 minimality of such a $\Gamma$-isometric dilation, the powers of the
 dilation operators will run over non-negative integers only.
\end{defn}

In the dilation theory of a single contraction (\cite{nazy}), it
is a notable fact that if $V$ is the minimal isometric dilation of
a contraction $T$, then $V^*$ is a co-isometric extension of $P$.
The other way is also true, i.e. if $V$ is a co-isometric
extension of $T$, then $V^*$ is an isometric dilation of $T^*$.
Here we shall see that an analogue holds for a
$\Gamma$-contraction.

\begin{prop}\label{easyprop3}
Let $(T,V)$ on $\mathcal K\supseteq \mathcal H$ be a
$\Gamma$-isometric dilation of a $\Gamma$-contraction $(S,P)$. If
$(T,V)$ is minimal, then $(T^*,V^*)$ is a $\Gamma$-co-isometric
extension of $(S^*,P^*)$. Conversely, if $(T^*,V^*)$ is a
$\Gamma$-co-isometric extension of $(S^*,P^*)$ then $(T,V)$ is a
$\Gamma$-isometric dilation of $(S,P)$.
\end{prop}
\begin{proof}
We first prove that $SP_{\mathcal H}=P_{\mathcal H}T$ and
$PP_{\mathcal H}=P_{\mathcal H}V$, where $P_{\mathcal H}:\mathcal
K \rightarrow \mathcal H$ is orthogonal projection onto $\mathcal
H$. Clearly
\[
\mathcal K=\overline{\textup{span}}\{ T^{m}V^n h\,:\; h\in\mathcal
H \textup{ and }m,n\in \mathbb N \cup \{0\} \}.
\]
Now for $h\in\mathcal H$ we have that
\begin{align*}
SP_{\mathcal H}(T^{m}V^n h)=S(S^{m}P^n h) &=S^{m+1}P^n h
\\&=P_{\mathcal H}(T^{m+1}V^n h)=P_{\mathcal H}T(T^{m}V^n h).
\end{align*}
Thus we have that $SP_{\mathcal H}=P_{\mathcal H}T$ and similarly
we can prove that
\[
PP_{\mathcal H}=P_{\mathcal H}V.
\]
Also for $h\in\mathcal H$ and
$k\in\mathcal K$ we have that
\begin{align*}
\langle S^*h,k \rangle =\langle P_{\mathcal H}S^*h,k
\rangle=\langle S^*h,P_{\mathcal H}k \rangle=\langle
h,SP_{\mathcal H}k \rangle &=\langle h,P_{\mathcal H}Tk \rangle
\\&=\langle T^*h,k \rangle .
\end{align*}
Hence $S^*=T^*|_{\mathcal H}$ and similarly $P^*=V^*|_{\mathcal
H}$. The converse part is obvious.
\end{proof}

Now we present the geometric construction of Sz.-Nagy dilation of
a $\Gamma$-contraction.

\begin{thm}\label{main-dilation-theorem}
 Let $(S,P)$ be a $\Gamma$-contraction defined on a Hilbert space $\mathcal H$. Let $F$ and $F_*$ be the
  fundamental operators of $(S,P)$ and its adjoint $(S^*,P^*)$ respectively. Let
 $\mathcal K_0=\cdots\oplus\mathcal D_P\oplus\mathcal D_P\oplus\mathcal D_P\oplus\mathcal H\oplus
 \mathcal D_{P^*}\oplus\mathcal D_{P^*}\oplus\mathcal D_{P^*}\oplus\cdots=l^2(\mathcal D_P)\oplus
 \mathcal H\oplus l^2(\mathcal D_{P^*})$. Consider the operator pair $(T_0,U_0)$ defined on $\mathcal K_0$ by
\begin{align*}
& T_0(\cdots,h_{-2},h_{-1},\underbrace{h_0},h_1,h_2,\cdots) \\&
=(\cdots,Fh_{-2}+F^*h_{-1},Fh_{-1}+F^*D_Ph_0-F^*P^*h_1,
\\& \quad \quad \underbrace{Sh_0+D_{P^*}
F_*h_1},F_*^*h_1+F_*h_2,F_*^*h_2+F_*h_3,\cdots)\\
& U_0(\cdots,h_{-2},h_{-1},\underbrace{h_0},h_1,h_2,\cdots) \\&
=(\cdots,h_{-2},h_{-1},D_Ph_0-P^*h_1,\underbrace{Ph_0+D_{P^*}h_1},h_2,h_3\cdots),
\end{align*}
where the $0$-th position of a vector in $\mathcal K_0$ has been indicated by an under brace. Then
$(T_0,U_0)$ is a minimal $\Gamma$-unitary dilation of $(S,P)$.
\end{thm}

 \begin{proof}
The matrices of $T_0$ with respect to the orthogonal
decompositions $l^2(\mathcal D_P)\oplus \mathcal H\oplus
l^2(\mathcal D_{P^*})$ and $\cdots\oplus\mathcal D_P\oplus\mathcal
D_P\oplus\mathcal D_P\oplus\mathcal H\oplus \mathcal
D_{P^*}\oplus\mathcal D_{P^*}\oplus\mathcal D_{P^*}\oplus\cdots$
of $\mathcal K_0$ and the matrix of $U_0$ with respect to the
decomposition $\cdots\oplus\mathcal D_P\oplus\mathcal
D_P\oplus\mathcal D_P\oplus\mathcal H\oplus \mathcal
D_{P^*}\oplus\mathcal D_{P^*}\oplus\mathcal D_{P^*}\oplus\cdots$
are the following:

\begin{eqnarray}
&T_0 = \left[
\begin{array}{ccc}
A_1 & A_2 & A_3\\
0 & S & A_4\\
0& 0& A_5  \end{array} \right] \notag \\& \label{2.3} =\left[
\begin{array}{ c c c c|c|c c c c}
\bm{\ddots}&\vdots &\vdots&\vdots   &\vdots  &\vdots& \vdots&\vdots&\vdots\\
\cdots&F&F^*&0  &0&  0&0&0&\cdots\\
\cdots&0&F&F^*  &0&  0&0&0&\cdots\\
%\hline
\cdots&0&0&F  &F^*D_P&  -F^*P^*&0&0&\cdots\\ \hline

\cdots&0&0&0   &S&   D_{P^*}F_*&0&0&\cdots\\ \hline

\cdots&0&0&0   &0&  F_*^*& F_*&0&\cdots\\
\cdots&0&0&0   &0&  0&F_*^*&F_*&\cdots\\
\cdots&0&0&0  &0&   0& 0&F_*^*&\cdots\\
\vdots&\vdots&\vdots&\vdots&\vdots&\vdots&\vdots&\vdots& \bm{\ddots}\\
\end{array} \right],
\end{eqnarray}

\vspace{5mm}

\begin{equation}\label{2.4}
U_0 = \left[
\begin{array}{ c c c c|c|c c c c}
\bm{\ddots}&\vdots &\vdots&\vdots   &\vdots  &\vdots& \vdots&\vdots&\vdots\\
\cdots&0&I&0  &0&  0&0&0&\cdots\\
\cdots&0&0&I  &0&  0&0&0&\cdots\\
%\hline
\cdots&0&0&0  &D_P&  -P^*&0&0&\cdots\\ \hline

\cdots&0&0&0   &P&   D_{P^*}&0&0&\cdots\\ \hline

\cdots&0&0&0   &0&  0& I&0&\cdots\\
\cdots&0&0&0   &0&  0&0&I&\cdots\\
\cdots&0&0&0  &0&   0& 0&0&\cdots\\
\vdots&\vdots&\vdots&\vdots&\vdots&\vdots&\vdots&\vdots&\bm{\ddots}\\
\end{array} \right].
\end{equation}

 \noindent The dilation space $\mathcal K_0$ is obviously the minimal unitary dilation space
 of the contraction $P$ and clearly the operator $U_0$ is the minimal unitary dilation of $P$.
 The space $\mathcal H$ can be embedded inside $\mathcal K_0$ by the canonical map
 $h \mapsto (\cdots,0,0,\underbrace{h},0,0,\cdots)$.
 The adjoint of $T_0$ and $U_0$ are defined in the following
way:
\begin{align*}
& T_0^*(\cdots,h_{-2},h_{-1},\underbrace{h_0},h_1,h_2,\cdots)\\&
=(\cdots,Fh_{-3}+F^*h_{-2},Fh_{-2}+F^*h_{-1},
\\& \quad \quad \underbrace{D_PFh_{-1}+S^*h_0},-PFh_{-1}+F_*^*D_{P^*}h_0+F_*h_1,F_*^*h_1+F_*h_2,\cdots)\\
& U_0^*(\cdots,h_{-2},h_{-1},\underbrace{h_0},h_1,h_2,\cdots)\\&
=(\cdots,h_{-3},h_{-2},\underbrace{D_Ph_{-1}+P^*h_0},-Ph_1,D_{P^*}h_0,h_1\cdots).
\end{align*}
To prove $(T_0,U_0)$ to be a minimal $\Gamma$-unitary dilation of
$(S,P)$ we have to show the following:
\begin{itemize}
\item[(i)]$(T_0,U_0)$ is a $\Gamma$-unitary \item[(ii)]
$(T_0,U_0)$ dilates $(S,P)$ \item[(iii)] the dilation $(T_0,U_0)$
is minimal.
\end{itemize}
For proving $(T_0,U_0)$ to be a $\Gamma$-unitary one needs to
verify, by virtue of Theorem \ref{G-unitary}-part (3), the
following:
\[
T_0U_0=U_0T_0,\, T_0=T_0^*U_0 \textup{ and } r(T_0)\leq 2.
\] Now
\begin{align*}
& T_0U_0(\cdots,h_{-2},h_{-1},\underbrace{h_0},h_1,h_2,\cdots) \\&
=T_0(\cdots,h_{-2},h_{-1},D_Ph_0-P^*h_1,\underbrace{Ph_0+D_{P^*}h_1},h_2,h_3\cdots)\\&
=(\cdots,Fh_{-1}+F^*D_Ph_0-F^*P^*h_1,\\& \qquad
(FD_P+F^*D_PP)h_0+(-FP^*+F^*D_PD_{P^*})h_1 -F^*P^*h_2,\\& \qquad
\underbrace{SPh_0+SD_{P^*}h_1+D_{P^*}F_*h_2},F_*^*h_2+F_*h_3,F_*^*h_3+F_*h_4,\cdots).
\end{align*}
Also
\begin{align*}
& U_0T_0(\cdots,h_{-2},h_{-1},\underbrace{h_0},h_1,h_2,\cdots)\\&
=U_0(\cdots,Fh_{-2}+F^*h_{-1},Fh_{-1}+F^*D_Ph_0-F^*P^*h_1,
\\& \quad \quad \underbrace{Sh_0+D_{P^*}
F_*h_1},F_*^*h_1+F_*h_2,F_*^*h_2+F_*h_3,\cdots)\\&
=(\cdots,Fh_{-1}+F^*D_Ph_0-F^*P^*h_1,D_PSh_0+(D_PD_{P^*}F_*-P^*F_*^*)h_1
\\& \qquad -P^*F_*h_2, \underbrace{PSh_0+(PD_{P^*}
F_*+D_{P^*}F_*^*)h_1+D_{P^*}F_*h_2},
\\& \qquad \qquad F_*^*h_2+F_*h_3,F_*^*h_3+F_*h_4,\cdots).
\end{align*}
In order to prove $T_0U_0=U_0T_0$ we have to prove the following
things:
\begin{itemize}
\item[($a_1$)]$D_PS=FD_P+F^*D_PP$, \item[($a_2$)]
$D_PD_{P^*}F_*-P^*F_*^*=-FP^*+F^*D_PD_{P^*}$, \item[($a_3$)]
$SD_{P^*}=D_{P^*}F_*^*+PD_{P^*}F_*$, \item[($a_4$)]
$F^*P^*=P^*F_*$.
\end{itemize}
($a_1$). Let $J=FD_P+F^*D_PP-D_PS$. Then $J$ is an operator from
$\mathcal H$ to $\mathcal D_P$. Since $F$ is the solution of
$S-S^*P=D_PXD_P$ we have that
\begin{align*}
D_PJ &= D_PFD_P+D_PF^*D_PP-D_P^2S \\& =
(S-S^*P)+(S^*-P^*S)P+(I-P^*P)S \\&=0.
\end{align*}
Clearly $\langle Jh,D_Ph_1 \rangle=\langle D_PJh,h_1 \rangle=0$
for all $h,h_1\in\mathcal H$ and hence $J=0$ which proves ($a_1$).

\noindent ($a_2$). It is enough to show that
$FP^*-P^*F_*^*=F^*D_PD_{P^*}-D_PD_{P^*}F_*$, where each side is
defined from $\mathcal D_{P^*}$ to $\mathcal D_P$.
\begin{align*}
& D_P(FP^*-P^*F_*^*)D_{P^*}\\&
=(D_PFD_P)P^*-P^*(D_{P^*}F_*^*D_{P^*}), \textup{ using the
relation } PD_P=D_{P^*}P\\& =(S-S^*P)P^*-P^*(S^*-SP^*)^*\\&
=SP^*-S^*PP^*-P^*S+P^*PS
\\&=(S^*-P^*S)(I-PP^*)-(I-P^*P)(S^*-SP^*)
\\&=(S^*-P^*S)D_{P^*}^2-D_{P}^2(S^*-SP^*)\\&=(D_PF^*D_P)D_{P^*}^2-D_P^2(D_{P^*}F_*D_{P^*}).
\end{align*}
For a proof of $PD_P=D_{P^*}P$ one can see chapter-I of
\cite{nazy}. Hence ($a_2$) is proved.\\

\noindent ($a_3$). Setting $J_1=D_{P^*}F_*^*+PD_{P^*}F_*-SD_{P^*}$
which maps $\mathcal D_{P^*}$ into $\mathcal H$ and using the same
argument as in the proof of ($a_1$), we can obtain $J_1D_{P^*}=0$
which proves ($a_3$).\\

\noindent ($a_4$). This follows from the fact that
$PF=F_*^*P|_{\mathcal D_P}$.\\

\textit{Proof of} $PF=F_*^*P|_{\mathcal D_P}$:\\
For $D_Ph \in
\mathcal D_P$ and $D_{P^*}h^{\prime}\in \mathcal D_{P^*}$, we have
that
\begin{align*}
\langle PFD_Ph,D_{P^*}h^{\prime} \rangle =\langle
D_{P^*}PFD_Ph,h^{\prime} \rangle &=\langle
PD_PFD_Ph,h^{\prime}\rangle\\
&=\langle P(S-S^*P)h,h^{\prime} \rangle \\&=\langle
(S-PS^*)Ph,h^{\prime} \rangle\\&=\langle
D_{P^*}{F_*}^*D_{P^*}Ph,h^{\prime} \rangle, \\& \text{since }
S^*-SP^*=D_{P^*}{F_*}D_{P^*}
\\&=\langle {F_*}^*PD_Ph,D_{P^*}h^{\prime} \rangle.
\end{align*}

Therefore $T_0U_0=U_0T_0$.\\

\noindent We now show that $T_0=T_0^*U_0$. We have that
\begin{align*}
& T_0^*U_0(\cdots,h_{-2},h_{-1},\underbrace{h_0},h_1,h_2,\cdots)
\\&
=T_0^*(\cdots,h_{-2},h_{-1},D_Ph_0-P^*h_1,\underbrace{Ph_0+D_{P^*}h_1},h_2,h_3\cdots)\\
&=( \cdots,Fh_{-2}+F^*h_{-1},Fh_{-1}+F^*D_Ph_0-F^*P^*h_1,
\\& \quad
\underbrace{(D_PFD_P+S^*P)h_0+(-D_PFP^*+S^*D_{P^*})h_1},\\& \quad
(-PFD_P+F_*^*D_{P^*}P)h_0+(PFP^*+F_*^*D_{P^*}^2)h_1+F_*h_{2},\\&
\quad F_*^*h_2+F_*h_3,\cdots ).
\end{align*}
Since $F$ is the fundamental operator of $(S,P)$ we have
$S=S^*P+D_PFD_P$. Therefore, in order to prove $T_0=T_0^*U_0$, we
need to show the following three steps:
\begin{itemize}
\item[($b_1$)]$D_{P^*}F_*=S^*D_{P^*}-D_PFP^*$ \item[($b_2$)]
$PFD_P=F_*^*D_{P^*}P$ \item[($b_3$)] $PFP^*+F_*^*D_{P^*}^2=F_*^*$.
\end{itemize}
For proving $(b_1)$ let us set $G=D_{P^*}F_*+D_PFP^*-S^*D_{P^*}$.
Obviously $G$ maps $\mathcal D_{P^*}$ into $\mathcal H$ and
\begin{align*}
GD_{P^*}&=D_{P^*}F_*D_{P^*}+D_PFP^*D_{P^*}-S^*D_{P^*}^2 \\&
=(S^*-SP^*)+(S-S^*P)P^*-S^*(I-PP^*), \text{ by } P^*D_{P^*}=D_PP^*
\\& =0,
\end{align*}
which proves $(b_1)$. The other two parts, $(b_2)$ and $(b_3)$,
follow from the facts that $PD_P =D_{P^*}P$ and
$PF=F_*^*P|_{\mathcal D_P}$.\\

In the matrix of $T_0$, $A_1$ on $l^2(\mathcal D_P)$ is same as
the multiplication operator $M_{F+F^*z}$ on $l^2(\mathcal D_P)$.
For $z=e^{i \theta}\in\mathbb{T}$ we have that
\begin{align*} \|F+F^*z\|&= \|F+e^{i \theta}F^*\|\\&=\|e^{-i
\theta/2}F+e^{i \theta/2}F^*\|\\&= \sup_{\|x\|\leq1}|\langle
(e^{-i \theta/2}F+e^{i \theta/2}F^*)x,x \rangle| \\& \qquad
\text{since } e^{-i \theta/2}F+e^{i \theta/2}F^* \text{ is
self-adjoint} \\& \leq \omega(F)+\omega(F^*) \\& \leq 2, \quad
[\textup{ since }\omega(F)\leq 1].
\end{align*} So by \textit{Maximum Modulus Principle}, $\|F+F^*z\|\leq2$
for all $z\in \overline{\mathbb{D}}$. Therefore
$\|A_1\|=\|M_{F+F^*z}\|=\|F+F^*z\|\leq2$. Similarly we can show
that $\|A_5\|\leq 2$. Also $\|S\|\leq 2$, because $(S,P)$ is a
$\Gamma$-contraction. Again by Lemma 1 of \cite{hong} we have that
$\sigma(T_0)\subseteq \sigma(A_1)\cup\sigma(S)\cup\sigma({A_5})$.
Therefore, $r(T_0)\leq 2$. Hence $(T_0,U_0)$ is a
$\Gamma$-unitary.

It is evident from the matrices of $T_0$ and $U_0$ that
$P_{\mathcal H}(T_0^mU_0^n)|_{\mathcal H}=S^mP^n$ for all
non-negative integers $m,n$ which proves that $(T_0,U_0)$ dilates
$(S,P)$. The minimality of the $\Gamma$-unitary dilation
$(T_0,U_0)$ follows from the fact that $\mathcal K_0$ and $U_0$
are respectively the minimal unitary dilation space and minimal
unitary dilation of $P$. Hence the proof is complete.
\end{proof}

An explicit $\Gamma$-isometric dilation of a $\Gamma$-contraction
was provided in \cite{tirtha-sourav} (see Theorem 4.3 in
\cite{tirtha-sourav}). Here we show that the $\Gamma$-isometric
dilation can easily be obtained as the restriction of the
$\Gamma$-unitary dilation described in the previous theorem.
\begin{cor}\label{gamma isometric dilation}
 Let $\mathcal N_0\subseteq \mathcal K_0$ be defined as $\mathcal N_0=\mathcal H\oplus l^2({\mathcal
 D_P})$. Then $\mathcal N_0$ is a common invariant subspace of $T_0,U_0$ and
 $(T^{\flat},V^{\flat})=(T_0|_{\mathcal N_0},U_0|_{\mathcal N_0})$
 is a minimal $\Gamma$-isometric dilation of $(S,P)$.
 \end{cor}
 \begin{proof}
It is evident from the matrix form of $T_0$ and $U_0$ (from the
previous theorem) that $\mathcal N_0=\mathcal H\oplus
l^2({\mathcal
 D_P})\mathcal= H\oplus \mathcal
D_P\oplus\mathcal D_P\oplus\cdots$ is a common invariant subspace
of $T_0$ and $U_0$. Therefore by the definition of
$\Gamma$-isometry, the restriction of $(T_0,U_0)$ to the common
invariant subspace $\mathcal N_0$, i.e. $(T^{\flat},V^{\flat})$ is
a $\Gamma$-isometry. The matrices of $T^{\flat}$ and $V^{\flat}$
with respect to the decomposition $\mathcal H\oplus \mathcal
D_P\oplus\mathcal D_P\oplus\cdots$ of $\mathcal N_0$ are the
following:
\[
T^{\flat}=\begin{bmatrix} S&0&0&0&\cdots\\F^*D_P&F&0&0&\cdots\\0&F^*&F&0&\cdots\\0&0&F^*&F&\cdots\\
\vdots&\vdots&\vdots&\vdots&\ddots \end{bmatrix}\;,\quad
V^{\flat}=\begin{bmatrix}
P&0&0&0&\cdots\\D_P&0&0&0&\cdots\\0&I&0&0&\cdots\\0&0&I&0&\cdots\\
\vdots&\vdots&\vdots&\vdots&\ddots \end{bmatrix}.
\]
It is obvious from the matrices of $T^{\flat}$ and $V^{\flat}$
that the adjoint of $(T^{\flat},V^{\flat})$ is a
$\Gamma$-co-isometric extension of $(S^*,P^*)$. Therefore by
Proposition \ref{easyprop3}, $(T^{\flat},V^{\flat})$ is a
$\Gamma$-isometric dilation of $(S,P)$. The minimality of this
$\Gamma$-isometric dilation follows from the fact that $\mathcal
N_0$ and $V^{\flat}$ are respectively the minimal isometric
dilation space and minimal isometric dilation of $P$. Hence the
proof is complete.
\end{proof}
\begin{rem}
The minimal $\Gamma$-unitary dilation $(T_0,U_0)$ described in
Theorem \ref{main-dilation-theorem} is the minimal
$\Gamma$-unitary extension of minimal $\Gamma$-isometric dilation
$(T^{\flat},V^{\flat})$ given in Corollary \ref{gamma isometric
dilation}. The reason is that if there is any $\Gamma$-unitary
extension $(T,U)$ of $(T^{\flat},V^{\flat})$ then $U$ is a unitary
extension of $V^{\flat}$ and $U_0$ is the minimal unitary
extension of $V^{\flat}$.
\end{rem}

\section{Functional Models}

Wold-decomposition breaks an isometry into two parts namely a
unitary and a pure isometry. A pure isometry $V$ is unitarily
equivalent to the Toeplitz operator $T_z$ on $H^2(\mathcal
D_{V^*})$. We have an analogous Wold-decomposition for
$\Gamma$-isometries in terms of a $\Gamma$-unitary and a pure
$\Gamma$-isometry (Theorem \ref{Gamma-isometry}, part-[2]). Again
Theorem \ref{G-unitary} tells us that every $\Gamma$-unitary is
nothing but the symmetrization of a pair of commuting unitaries.
Therefore a standard model for pure $\Gamma$-isometries gives a
complete picture of a $\Gamma$-isometry. In \cite{tirtha-sourav1},
a functional model for pure $\Gamma$-contractions has been
described. When in particular we are concerned about pure
$\Gamma$-isometries, it requires a much simpler effort to
establish the model.

\begin{thm}\label{model1}
Let $(\hat S,\hat P)$ be a commuting pair of operators on a
Hilbert space $\mathcal H$. If $(\hat S,\hat P)$ is a pure
$\Gamma$-isometry then there is a unitary operator $U:\mathcal H
\rightarrow H^2(\mathcal D_{{\hat P}^*})$ such that
\[
\hat S=U^*T_{\varphi}U, \text{ and } \hat P=U^*T_zU, \text{ where
} \varphi(z)= {\hat F}_*^*+{\hat F}_*z.
\]
Here ${\hat F}_*$ is the fundamental operator of $({\hat
S}^*,{\hat P}^*)$. Conversely, every such pair $(T_{A+A^*z},T_z)$
on $H^2(E)$ for some Hilbert space $E$ with $\omega(A)\leq 1$ is a
pure $\Gamma$-isometry.
\end{thm}

\begin{proof}
First let us suppose that $(\hat S,\hat P)$ is a pure
$\Gamma$-isometry. Then $\hat P$ is a pure isometry and can be
identified with $T_z$ on $H^2(\mathcal D_{{\hat P}^*})$.
Therefore, there is a unitary $U$ from $\mathcal H$ onto
$H^2(\mathcal D_{{\hat P}^*})$ such that $\hat P=U^*T_zU$. Since
$\hat S$ is a commutant of $\hat P$, there exists $\varphi\in
H^{\infty}(\mathcal L(D_{{\hat P}^*}))$ such that
$T=U^*T_{\varphi}U$. As $(T_{\varphi},T_z)$ is a
$\Gamma$-isometry, by the relation $T_{\varphi}=T_{\varphi}^*T_z$
(see Theorem \ref{Gamma-isometry}), we have that
\[
\varphi(z)=A+A^*z,\; \textup{ for some } A\in\mathcal L(\mathcal
D_{V^*}).
\]
Also $\|T_{\varphi}\|=\|\varphi\|_{\infty}\leq 2$. Therefore, for
any real $\theta$,
\[
\|A+A^*e^{i\theta}\| = \|Ae^{-i\theta /2}+A^*e^{i\theta
/2}\|=\|2\textup{Re}(e^{-i\theta /2}A)\|\leq 2.
\]
Therefore,
$\omega(A)\leq 1$ by Lemma \ref{basicnrlemma}. It is evident from
the proof of Theorem \ref{fundamentalop} that if
$(T_{A+A^*z},T_z)$ is a $\Gamma$-isometry then $A^*$ is the
fundamental operator of the $\Gamma$-co-isometry
$(T_{A+A^*z}^*,T_z^*)$. Denoting by ${\hat F}_*$, the fundamental
operator of $({\hat S}^*,{\hat P}^*)$, we have that $\hat
S=U^*T_{{\hat F}_*^*+{\hat F}_*z}U$.

The proof to the converse is simple. The fact that
$(T_{A+A^*z},T_z)$ on $H^2(E)$ is a $\Gamma$-isometry, when
$\omega(A)\leq 1$, follows from Theorem \ref{Gamma-isometry},
Part-(3). Moreover, since $T_z$ is pure isometry,
$(T_{A+A^*z},T_z)$ is a pure $\Gamma$-isometry.

\end{proof}

The following result of one variable dilation theory is necessary
for the proof of the model theorem for a $\Gamma$-contraction. We
present a proof of it due to lack of a good reference.
\begin{prop}\label{easyprop1}
If $T$ is a contraction and $V$ is its minimal isometric dilation
then $T^*$ and $V^*$ have defect spaces of same dimension.
\end{prop}
\begin{proof}
Let $T$ and $V$ be defined on $\mathcal H$ and $\mathcal K$. Since
$V$ is the minimal isometric dilation of $T$ we have
\[
\mathcal K=\overline{\text{span}} \{p(V)h:\; h\in\mathcal H \text{
and }p \text{ is any polynomial in one variable }\}.
\]
The defect spaces of $T^*$ and $V^*$ are respectively $\mathcal
D_{T^*}=\overline{\textup{Ran}}\;(I-TT^*)^{\frac{1}{2}}$ and
$D_{V^*}=\overline{\textup{Ran}}\; (I-VV^*)^{\frac{1}{2}}$. Let
$\mathcal N= \overline{\textup{Ran}}\;
(I-VV^*)^{\frac{1}{2}}|_{\mathcal H}$. For $h\in\mathcal H$ and
$n\geq 1$, we have
\[
(I-VV^*)V^nh=V^nh-VV^*V^nh=0, \text{ as } V \text{ is an
isometry}.
\]
Therefore, $(I-VV^*)p(V)h=p(0)(I-VV^*)h$ for any
polynomial $p$ in one variable. So $(I-VV^*)k\in\mathcal N$ for
any $k\in\mathcal K$. This shows that
$\overline{\textup{Ran}}(I-VV^*)\subseteq \mathcal N$ and
hence $\overline{\textup{Ran}}(I-VV^*)=\mathcal D_{V^*}=\mathcal N$.\\

We now define for $h\in\mathcal H$,
\begin{align*}
L:\; & \text{Ran}(I-TT^*)^{\frac{1}{2}}\,\rightarrow\,
\text{Ran}(I-VV^*)^{\frac{1}{2}}
\\& (I-TT^*)^{\frac{1}{2}}h \mapsto
(I-VV^*)^{\frac{1}{2}}h.
\end{align*}
We prove that $L$ is an
isometry. Since $V^*$ is co-isometric extension of $T^*$,
$TT^*=P_{\mathcal H}VV^*|_{\mathcal H}$ and thus we have $
(I_{\mathcal H}-TT^*)=P_{\mathcal H}(I_{\mathcal
K}-VV^*)|_{\mathcal H}$, that is, $D_{P^*}^2=P_{\mathcal
H}D_{V^*}^2|_{\mathcal H}.$ Therefore, for $h\in\mathcal H$,
\begin{align*}
\|D_{T^*}h\|^2=\langle D_{P^*}^2h,h \rangle=\langle P_{\mathcal
H}D_{V^*}^2h,h \rangle =\langle D_{V^*}^2h,h
\rangle=\|D_{V^*}h\|^2,
\end{align*}
and $L$ is an isometry and this can clearly be extended to a
unitary from $\mathcal D_{T^*}$ to $\mathcal D_{V^*}$. Hence
proved.
\end{proof}

The following is the model theorem of a $\Gamma$-contraction and
is another main result of this section. This can be treated as a
concrete form of the model given by Agler and Young (Theorem
\ref{model}) in the sense that we have specified the model space
and model operators.

\begin{thm}\label{model2}
 Let $(S,P)$ be a $\Gamma$-contraction on a Hilbert space $\mathcal H$.
 Let $(T,V)$ on $\mathcal K_*=\mathcal H\oplus \mathcal D_{P^*}\oplus\mathcal D_{P^*}\oplus
 \cdots$ be defined as
 \[
 T=\begin{bmatrix} S&D_{P^*}F_*&0&0&\cdots\\ 0&F_*^*&F_*&0&\cdots\\
 0&0&F_*^*&F_*&\cdots \\ 0&0&0&F_*^*&\cdots\\ \vdots&\vdots&\vdots&\vdots&\ddots  \end{bmatrix}
 \text{ and }
V=\begin{bmatrix} P&D_{P^*}&0&0&\cdots\\0&0&I&0&\cdots\\0&0&0&I&\cdots \\
 0&0&0&0&\cdots\\\vdots&\vdots&\vdots&\vdots&\ddots \end{bmatrix},
 \]
 where $F_*$ is the fundamental operator of $(S^*,P^*)$. Then
 \begin{enumerate}
 \item [(1)] $(T,V)$ is a $\Gamma$-co-isometry, $\mathcal H$ is a
 common invariant subspace of $T,V$ and $T|_{\mathcal H}=S$ and
 $V|_{\mathcal H}=P$; \item [(2)] there is an orthogonal decomposition $\mathcal K_*=\mathcal K_1\oplus \mathcal K_2$
 into reducing subspaces of $T$ and $V$ such that $(T|_{\mathcal K_1},V|_{\mathcal K_1})$ is a
 $\Gamma$-unitary and $(T|_{\mathcal K_2},V|_{\mathcal K_2})$ is a
 pure $\Gamma$-co-isometry; \item [(3)] $\mathcal K_2$ can be
 identified with $H^2(\mathcal D_V)$, where $D_V$ has same dimension
 as that of $\mathcal D_P$. The operators $T|_{\mathcal K_2}$ and $V|_{\mathcal
 K_2}$ are respectively unitarily equivalent to $T_{B+B^*\bar z}$ and $T_{\bar
 z}$ defined on $H^2(\mathcal D_V)$, $B$ being the fundamental
 operator of $(T,V)$.
 \end{enumerate}
 \end{thm}
 \begin{proof}
 It is evident from Corollary \ref{gamma isometric dilation} that
 $(T^*,V^*)$ is minimal $\Gamma$-isometric dilation of $(S^*,P^*)$,
 where $V^*$ is the minimal isometric dilation of $P^*$.
 Therefore by Proposition \ref{easyprop3}, $(T,V)$ is
 $\Gamma$-co-isometric extension of $(S,P)$. So we have that $\mathcal H$ is a
 common invariant subspace of $T$ and $V$ and $T|_{\mathcal H}=S,\; V|_{\mathcal H}=P$.
 Again since $(T^*,V^*)$ is a $\Gamma$-isometry, by Theorem \ref{Gamma-isometry} part-(2),
 there is an orthogonal decomposition $\mathcal K_*=\mathcal K_1\oplus\mathcal K_2$ into
 reducing subspaces of $T$ and $V$ such that $(T|_{\mathcal K_1},V|_{\mathcal K_1})$ is a
 $\Gamma$-unitary and $(T|_{\mathcal K_2},V|_{\mathcal
 K_2})$ is a pure $\Gamma$-co-isometry. If we denote $(T|_{\mathcal K_1},V|_{\mathcal
 K_1})$ by $(T_1,V_1)$ and $(T|_{\mathcal K_2},V|_{\mathcal K_2})$
 by $(T_2,V_2)$ then with respect to the orthogonal decomposition $\mathcal K_*=\mathcal K_1\oplus \mathcal K_2$
 we have
 \[
 T=\begin{bmatrix}T_1&0\\0&T_2 \end{bmatrix}\;,\quad V=\begin{bmatrix}V_1&0\\0&V_2 \end{bmatrix}.
 \]
 The fundamental equation $T-T^*V=D_VXD_V$ clearly becomes
 \[
 \begin{bmatrix}T_1-T_1^*V_1 &0\\0&T_2-T_2^*V_2 \end{bmatrix}
 =\begin{bmatrix}0&0\\0& D_{V_2}X_2D_{V_2} \end{bmatrix},\quad X=\begin{bmatrix} X_1\\X_2\end{bmatrix}.
 \]
 Since $\mathcal D_V=\mathcal D_{V_2}$, the above form of the
 fundamental equation shows that $(T,V)$ and $(T_2,V_2)$ have the
 same fundamental operator. Now we apply Theorem \ref{model1} to
 the pure $\Gamma$-isometry $(T_2^*,V_2^*)=(T^*|_{\mathcal K_2},V^*|_{\mathcal
 K_2})$ and get the following:
 \begin{enumerate}
 \item[(i)] $\mathcal K_2$ can be identified with $H^2(\mathcal D_{V_2})=H^2(\mathcal
 D_V)$; \item[(ii)] $T_2^*$ and $V_2^*$ can be identified with the
 Toeplitz operators $T_{B^*+Bz}$ and $T_z$ respectively defined on
 $H^2(\mathcal D_V)$, $B$ being the fundamental operator of
 $(T,V)$.
 \end{enumerate}
 Therefore, $T|_{\mathcal K_2}$ and $V|_{\mathcal
 K_2}$ are respectively unitarily equivalent to $T_{B+B^*\bar z}$ and $T_{\bar
 z}$ defined on $H^2(\mathcal D_V)$. Also since $V^*$ is the minimal isometric dilation of $P^*$ by
 Proposition \ref{easyprop1}, $\mathcal D_V$ and $\mathcal D_P$
 have same dimension.
\end{proof}

\noindent \textbf{Acknowledgement.} The author would like to thank
Orr Shalit for his invaluable comments on this article. Moreover,
the author is grateful to Orr Shalit for providing warm and
generous hospitality at Ben-Gurion University, Be'er Sheva,
Israel.

\end{document}